\documentclass{article}
\usepackage{graphicx} 
\usepackage{xcolor}
\usepackage{enumitem}
\usepackage{svg}
\usepackage{algorithm}
\usepackage{multicol}
\usepackage{graphicx}      
\usepackage{amsfonts}
\usepackage{amsmath,mathtools,amssymb,amsfonts}

\newcommand{\norm}[1]{\left\lVert#1\right\rVert}
\newcommand{\xddots}{%
  \raise 4pt \hbox {.}
  \mkern 6mu
  \raise 1pt \hbox {.}
  \mkern 6mu
  \raise -2pt \hbox {.}
}

\DeclareUnicodeCharacter{221E}{\ensuremath{\infty}}
\DeclareMathOperator*{\Minimize}{Minimize:}

\newtheorem{theorem}{\bf{Theorem}}

\newtheorem{assumption}{\bf{Assumption}}

\newtheorem{remark}{\bf{Remark}}

\newenvironment{proof}{\textit{Proof:}}{\hfill$\square$}

\title{Decentralized Droop-based Finite-Control-Set Model Predictive Control of Inverter-based Resources in Islanded AC Microgrid}
\author{Ayobami Olajube, Koto Omiloli, Satish Vedula, and \\ Olugbenga Moses Anubi}

\date{}

\begin{document}

\maketitle
\begin{centering}
Department of Electrical and Computer Engineering, the Center for Advanced Power Systems, Florida State University 
E-mail: \{aolajube, kao23a\ , svedula\ , oanubi\}@fsu.edu
\end{centering}

\section{Abstract}
This paper presents an improved droop control method to ensure effective power sharing, voltage regulation, and frequency stabilization of inverter-based resources (IBRs) connected in parallel in an islanded AC microgrid. In the contemporary droop control algorithm, the distance between connected inverters affects the effectiveness of the active power-frequency and the reactive power-voltage droop characteristics which results in poor power sharing at the primary level of the microgrid. That is, high impedance emanating from long transmission lines results in instability, poor voltage tracking, and ineffective frequency regulation. Hence, in this work, we use a finite-control-set model predictive controller (FCS-MPC) in the inner loop, which gives efficient voltage tracking, good frequency regulation, and faster performance response. FCS-MPC is easy to implement in fast switching converters and does not suffer from computational burden unlike the continuous-set MPC and is also devoid of issues of multiple-loop, parameter variation, and slow response associated with conventional droop control methods. We derived the condition for bounded stability for the FCS-MPC and the proposed method is tested via a numerical simulation on three IBRs. The results show effective power sharing, capacitor voltage tracking, and efficient frequency regulation with reduced oscillations to changes in load. 

\textbf{Keywords:} Droop control, finite-control-set model predictive control, active power sharing, reactive power sharing, voltage regulation, inverter-based resources, islanded microgrid.

\section{Introduction}
Traditional Microgrids (MGs) are a conglomeration of numerous fossil-fuel-driven power generation elements like synchronous machines (SMs). However, in recent times, an increase in penetration of inverter-based resources (IBRs) such as solar and wind generation has caused a shift in power generation trends in the existing infrastructure. The addition of energy storage elements (ESEs) further upped the choice of power allocation in the MGs in terms of the available degree of freedom. Based on their configuration and the mode of operation MGs are classified into two categories known as \emph{grid-connected} and \emph{islanded} modes \cite{Hatziargyriou_2007}. Control of the traditional MG employs a hierarchical control structure \cite{Guerrero_2011}, with droop-control being the lowest level followed by the power and the energy management (EM) layers. The impending integration of IBRs into the traditional MGs is giving rise to structural and operational changes in the form of transitioning away from the SMs which are ramp-limited due to their inertial properties. However, the low inertia and the fast switching properties of the IBRs pose a major challenge in the form of assessing their stability and robustness \cite{Milano2018FoundationsAC}. One such method that is widely embraced to restore the inertia to the grid is to introduce \emph{virtual inertia} \cite{Anubi_2022,cady2017distributed}.
 
Regardless of the inertial properties of the SMs and the IBRs, the control objective remains unequivocal in terms of balancing the active and reactive power in the MGs while simultaneously maintaining the grid frequency around the desired operating point. Droop control is the method widely employed in addressing the control of the power and the frequency namely \emph{voltage droop} and \emph{frequency droop} \cite{silva2021string}. The IBRs are operated in parallel mode due to high reliability and flexibility in terms of operation. However, the application of droop control in such a parallel setting is not suitable while sharing a nonlinear load due to the presence of harmonic currents \cite{khaledian2017analysis}. Moreover, the frequency and amplitude deviation in the IBRs droop settings are dependent on the load scenarios \cite{Guerrero_2013}. Addition of \emph{virtual} control loops such as adding a virtual resistance into the droop methods and decentralizing the control (consensus-based) are proposed to address the complications arising from the presence of current harmonics \cite{Lu_2018}. Although Communication-based and consensus control algorithms offer efficient voltage regulation and effective power sharing among connected IBRs, the methods suffer from the huge cost and reliability of communication links. On the other hand, droop control does not require communication between IBRs to effect power sharing and voltage regulation \cite{9612101}. Multiple loops of conventional droop techniques often result in poor power-sharing due to parameter variations, especially in the current loop that generates PWM signals for the inverters \cite{6254006}. 

Thus, in this paper, we utilize FCS-MPC in the inner loop to directly generate optimal switching signals for the IBRs. Finite-control-set model predictive control (FCS-MPC) has emerged as a promising discrete-time model-based nonlinear predictive control algorithm applicable to fast-switching power electronics dominated systems. Unlike the continuous-set MPC that requires modulators and a longer horizon to generate the pulse-width modulated (PWM) signals, FCS-MPC eliminates the need for modulators and accounts for nonlinearities in the power converter dynamics \cite{ bayhan2016model}. It uses the horizon of h = 1, which significantly reduces the computational burden that may result from very low sampling time in high-frequency converters.    
To this end, we develop a decentralized droop-based FCS-MPC-based control algorithm to ensure adequate active and reactive power sharing between the parallel connected IBRs and the accurate tracking of the filter output filter capacitor voltage. The droop control method provides a reference current for the FCS-MPC algorithm which in turn generates the PWM signal to switch the parallel connected IBRs. We examine the behaviors of the IBRs with respect to their droop coefficients. Also, the effect of sudden changes in the load on the droop characteristics of the inverters is examined. Consequently, the key contributions in this paper are:
\begin{enumerate}
    \item The unique stability analysis for finite-set model predictive control (FCS-MPC) algorithm.
    \item The development of the decentralized droop-based finite-control-set model predictive control of several IBRs to regulate the AC grid voltage error to zero without using PWM modulators. 
    \item The evaluation of the performance of the droop-based FCS-MPC on both P-f and Q-V droop characteristics through simulations on three parallel connected IBRs.
     
\end{enumerate}
The rest of the paper is organized as follows: section II shows the preliminaries and notations used throughout this paper, section III shows the model development of decentralized inverter-based resources in an islanded AC grid, section IV depicts the control design and algorithm development, section V enumerates the numerical simulation results and the conclusion follows in section VI.

\section{Notations and Preliminaries}\label{Not_Pre}
In this paper, we use $\mathbb{R}^n$, $\mathbb{R}^{n\times m}$ to denote the space of real numbers, real vectors of length $n$, and real matrices of $n$ rows, and $m$ columns, respectively.
$\mathbb{R}_+$ indicates a set of positive real numbers.
$X^\top$ denotes the transpose of the quantity $X$.
Normal-face lower-case letters ($x\in\mathbb{R}$) are used to represent real scalars, bold-face lower-case letters ($\mathbf{x}\in\mathbb{R}^n$) represent vectors, while normal-face upper case ($X\in\mathbb{R}^{n\times m}$) represents matrices. $X\succ 0 \hspace{1mm} (\succeq0)$ denotes a positive definite (semi-definite) matrix. Furthermore, $J=\begin{bmatrix}
    0&1\\
    -1&0\\
    \end{bmatrix}$ is a skew-symmetric matrix, $M \otimes N =\begin{bmatrix}
    M_{11}N&M_{12}N\\
    M_{21}N&M_{22}N\\
    \end{bmatrix}$ denotes the Kronecker product of matrices $M$ and $N$, $\mathbf{1}_n$ is a column vector of ones, and the identity matrix of size $n$ is denoted by $I_n$. 
    
\subsection{Stability Analysis of Finite-Set Predictive Control System}
Consider a discrete-time system described by
\begin{equation}
      \textbf{x}(k+1) = f(\textbf{x}(k),\textbf{u}(k)) 
\end{equation}
Where $\textbf{x}(k) \in \mathcal{X} \subset \mathbb{R}^{n}$ is the system state and $\textbf{u}(k) \in \mathcal{U} \subset \mathbb{R}^{m}$ is the control input. If we assume the equilibrium point exists at the origin, we have $f(0,0) = 0$.
Suppose there exists positive constants $\alpha_1$, $\alpha_2$, $\alpha_3$, $\alpha_4$ and a positive definitive function $V: \mathbb{R}^n \rightarrow \mathbb{R}$ such that 
\begin{equation}\label{eqn:V}
\begin{array}{c}
\alpha_1\norm{\textbf{x}}^2 \leq V(\textbf{x}) \leq \alpha_2\norm{\textbf{x}}^2 \\
V(f(\textbf{x},\bar{u}(\textbf{x}))-V(\textbf{x}) \leq -\alpha_3 \norm{\textbf{x}}^2 \\
|V(\textbf{x})-V(\textbf{y})| \leq \underbrace{\alpha_4 (\norm{\textbf{x}}+\norm{\textbf{y}})}_{2\alpha_4\bar{\textbf{z}}}\norm{\textbf{x}-\textbf{y}} \leq 2\alpha_4\bar{\textbf{z}}\norm{\textbf{x}-\textbf{y}},
\end{array}
\end{equation}
where
\begin{equation*}
  \bar{\textbf{z}}=\sup_{\textbf{x} \in \mathcal{X}}\norm{\textbf{x}}  
\end{equation*}
for all $\textbf{x} \in \mathcal{X}\subset \mathbb{R}^{n}$, where $\bar{u}:\mathbb{R} \rightarrow \mathcal{U}$ is the associated feedback control and $\mathcal{U} \subset \mathbb{R}^{m}$ and $\mathcal{X} \subset \mathbb{R}^{n}$ are compact subsets. 
It follows that the origin of (1) is asymptotically stable under the control law $\bar{\textbf{u}} \in \mathcal{U}$ and $\mathcal{X}$ is positively invariant with respect to the closed loop dynamics.

Now let $\textbf{u}_s =\left\{u_1, u_2, \ldots, u_M\right\} \subseteq \mathcal{U}$ be a finite sampling of the compact set $\mathcal{U}$ such that 
\begin{equation}
\sup _{\textbf{u} \in \mathcal{U}} \min _{1 \ll i \leq M}\left\|\textbf{u}-\textbf{u}_i\right\|_2 \leq \epsilon
\end{equation}
for some $\epsilon > 0$. Here, $\epsilon$ is the quantization error. We are interested in the stability of (1) using the control law 
\begin{equation}
\textbf{u}(x)=\underset{1 \leq i \leq M}{\operatorname{argmin}}\left\|\bar{\textbf{u}}(\textbf{x})-\textbf{u}_i\right\|_2
\end{equation}
Clearly, the projection error $\left\|{\textbf{u}}(\textbf{x})-\bar{\textbf{u}}(\textbf{x})\right\|_2 \leq \epsilon$ follows from (3).
\begin{theorem}
    Suppose there exists a positive definite function $V:\mathbb{R}^n\rightarrow\mathbb{R}$ satisfying the inequalities in \eqref{eqn:V} for some $\alpha_1,\alpha_2,\alpha_3,\alpha_4\in\mathbb{R}_+$ such that
    \begin{align}
        \frac{\alpha_3}{\alpha_1}<2,
    \end{align}
    Then, the origin of the closed loop system $\textbf{x}(k+1) = f(\textbf{x}(k), \textbf{u}(k))$ under the finite-set control law in (4) is stable and 
\begin{equation*}
\lim _{k \rightarrow \infty}\|\textbf{x}(k)\| \leq \sqrt{\frac{2L_v \alpha_4 \bar{\textbf{z}}}{\alpha_3}\epsilon},
\end{equation*}
where $L_v$ is a Lipschitz constant of the function $f(x)$ in (1) and  $\bar{\textbf{z}}$ is related to compact set $\mathcal{X} \subset \mathbb{R}^n$ as given in (2).
\end{theorem}
\begin{proof}
    Let us consider a lyapunov function V(\textbf{x}) and a Lipschitz continuous function f(\textbf{x}) such that:
\begin{align*}
     V(f(\textbf{x},\textbf{u}(\textbf{x})))-V(\textbf{x}) &= V(f(\textbf{x},\textbf{u}(\textbf{x})))+ \\ &
    V(f(\textbf{x},\bar{\textbf{u}}(\textbf{x})))+V(f(\textbf{x},u(\textbf{x}))) \\ &
    -V(f(\textbf{x},\bar{\textbf{u}}(\textbf{x})))-V(\textbf{x})\\
     &\leq L_v\norm{f(\textbf{x},u(\textbf{x}))-f(\textbf{x},\bar{\textbf{u}}(\textbf{x}))} &\\- \alpha_3\norm{\textbf{x}}^2 \\
    &\leq 2 L_v \alpha_4 \bar{\textbf{z}}\norm{\textbf{u}(\textbf{x})-\bar{\textbf{u}}(\textbf{x})}-\alpha_3\norm{\textbf{x}}^2 \\
     &\leq 2\epsilon L_v \alpha_4 \bar{\textbf{z}} - \alpha_3\norm{\textbf{x}}^2. 
\end{align*}
For all $\textbf{x} \in \mathcal{X}$. Let $V_k \triangleq V(\textbf{x}(k))$, thus 
\begin{align*}
     V_{k+1}-V_k &\leq 2\epsilon L_v \alpha_4 \bar{\textbf{z}} - \alpha_3\norm{\textbf{x}_k}^2 \\ &
    \leq 2\epsilon L_v \alpha_4 \bar{\textbf{z}} - \frac{\alpha_3}{\alpha_1}V_k.
\end{align*}
This implies that 
\begin{equation*}
     V_{k+1} \leq (1-\frac{\alpha_3}{\alpha_1})V_k + 2\epsilon L_v \alpha_4 \bar{\textbf{\textbf{z}}}     
\end{equation*}
\begin{equation*}
    V_k \leq (1-\frac{\alpha_3}{\alpha_1})^kV_o + 2\epsilon L_v \alpha_4 \bar{\textbf{z}} \sum_{i=0}^{k-1} (1-\frac{\alpha_3}{\alpha_1})^i.
\end{equation*}
Hence, $V(\textbf{x}_k)$ is bounded and, assuming bounded initial conditions
\begin{equation*}
    \lim _{k \rightarrow \infty} V(\textbf{x}_k) \leq \frac{2L_v\alpha_1 \alpha_4 \bar{\textbf{z}} }{\alpha_3}\epsilon,
\end{equation*}
which implies that
\begin{equation*}
\lim _{k \rightarrow \infty}\|\textbf{x}(k)\| \leq \sqrt{\frac{2L_v \alpha_4 \bar{\textbf{z}}}{\alpha_3}\epsilon}
\end{equation*}
\end{proof}

\subsection{Communicationless design of Decentralized IBRs}
We consider an autonomous AC microgrid comprising inverter-based resources (IBRs) connected in parallel to supply power to constant power loads. It is represented as a graph $\mathcal{G} = (\mathcal{V},\mathcal{E}, \mathcal{B}_{ik})$, where $\mathcal{V} = \{v| v_i \in \mathbb{R}^n\}$ represents a set of IBRs and $\mathcal{E} \in \mathcal{V} \times \mathcal{V} = \{1,....,m\}$ is a set of transmission lines connecting the IBRs. The incident matrix $\mathcal{B}_{ik} \in \mathbb{R}^{n\times m}$ showing the interconnections and current flow between the IBRs is shown below:
\begin{equation*}
\mathcal{B}_{i k}= \begin{cases}+1, & \text { if IBR } i \text { starts \textit{i} to } k, \\ -1, & \text { if IBR } i \text { ends \textit{i} to } k, \\ 0, & \text { otherwise. }\end{cases}
\end{equation*}
When current flows from the i-th IBR to the k-th IBR, we assign a +1 sign and when current flows in the reverse direction, -1 is assigned. But, when there is no transmission line between any two inverters, we use zero.

\section{Model Development}
A hypothetical autonomous AC microgrid powered by IBRs is shown in Figure~\ref{fig:figa}. Each IBR is connected to the load through an LC filter, which helps to dampen harmonics and ensure accurate extraction of the fundamental frequency. An IBR shares current with the neighboring IBR through the resistive-inductive transmission line.  

\begin{remark}
    \it{The decentralized IBR dynamics in~(\ref{dq0_eqn}) is a balanced and symmetrical system. Hence, the zero component of the dq0 rotating reference frame is zero. Then, we will stick to dq components only for the rest of this paper.}
\end{remark}

The following important assumptions are made to establish the generalized state-space representation of the system shown in Figure \ref{fig:figa}:
\begin{assumption}
  \it{For n homogeneous number of IBRs, switching losses are neglected during power transformation from DC to AC.}
\end{assumption}

\begin{assumption}
  \it{The DC power supply to each IBR is stiff and the output voltages $\mathbf{u}_{d q}$ are local and unique to each inverter \cite{8810506}.} 
\end{assumption}

\begin{assumption}
  \it{The transmission line between each IBR is highly inductive (X$\gg$R) \cite{laaksonen2005voltage}. The power angle or phase shift angle under this condition becomes very small.}
\end{assumption}
\begin{assumption}
  \it{The load currents $\textbf{i}_{L{dq}}$ at the PCC are unknown, bounded, and obtainable from load voltage measurements at each bus.}
\end{assumption}

\begin{figure}[h]
    \centering \includegraphics[width=0.45\textwidth]{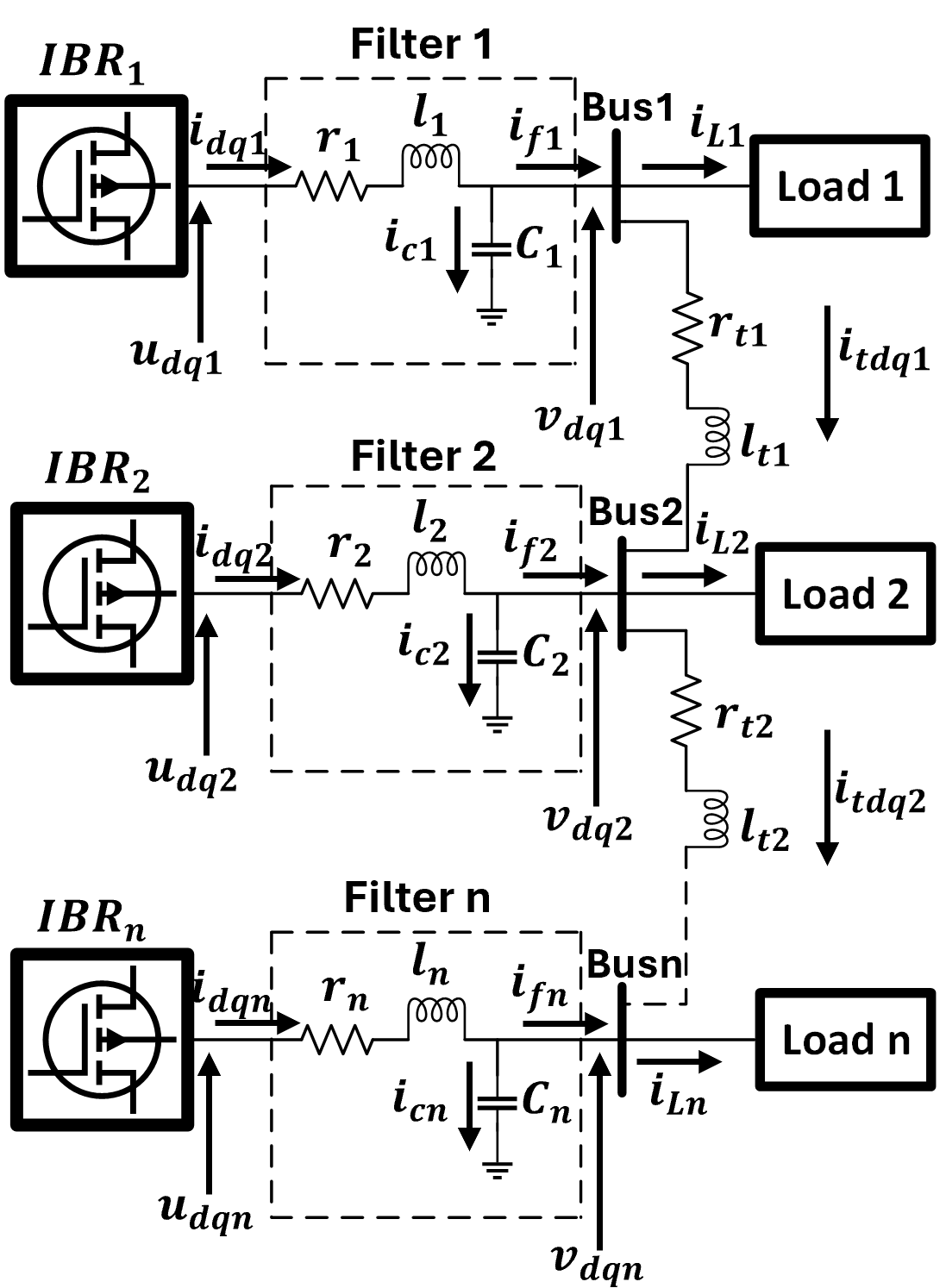}
    \caption{Decentralized Inverter-based Resources}
    \label{fig:figa}
\end{figure}

The model of the decentralized scheme shown in Figure 1 in \textit{dq0} reference frame is given as:
\begin{equation}
\begin{aligned}\label{dq0_eqn}
& L\frac{d \mathbf{i}_{dq}}{dt}=- \mathbf{v}_{\mathrm{dq}}-(\omega_g (J\otimes L)+ RI_n)\mathbf{i}_{\mathrm{dq}}+\mathbf{u}_{\mathrm{dq}}, \\
& C_s\frac{d \mathbf{v}_{d q}}{d t}=-\omega_g (J\otimes C_s)\mathbf{v}_{\mathrm{dq}}+ \mathbf{i}_{\mathrm{dq}}+ \mathcal{B}_{\mathrm{ik}} \mathbf{i}_{\mathrm{tdq}}- \mathbf{i}_{Ldq}, \\
& L_t\frac{d \mathbf{i}_{t{d q}}}{d t}=- \mathcal{B}_{\mathrm{ik}}^\top \mathbf{v}_{\mathrm{dq}}-(\omega_g (J\otimes L_t)+ RI_m)\mathbf{i}_{\mathrm{tdq}},
\\
& \mathbf{y}_{\mathrm{dq}}= [\mathbf{v}_{\mathrm{dq}};\mathbf{i}_{\mathrm{dq}}]
\end{aligned}
\end{equation}
where $\mathbf{i}_{dq} = [\mathbf{i}_d^\top,  \mathbf{i}_q^\top]^\top, \mathbf{v}_{dq}= [\mathbf{v}_d^\top,  \mathbf{v}_q^\top]^\top, \mathbf{u}_{dq}= [\mathbf{u}_d^\top,  \mathbf{u}_q^\top]^\top, \newline
\mathbf{y}_{dq}= [\mathbf{y}_d,  \mathbf{y}_q]^\top \in \mathbb{R}^{2n}$,  $\mathbf{i}_{L{d q}}= [\mathbf{i}_{Ld}^\top,  \mathbf{i}_{Lq}^\top]^\top \in \mathbb{R}^{2n}$ and $\mathbf{i}_{t{d q}}= [\mathbf{i}_{td}^\top,  \mathbf{i}_{tq}^\top]^\top \in \mathbb{R}^{2m}$. Then, the filter resistance matrix becomes $R = \textsf{diag}(r_1, r_2,\hdots,r_n)$, filter inductance matrix  $L=\textsf{diag}(l_1,l_2,\hdots,l_n)$, the filter capacitor matrix $C_s=\textsf{diag}(C_1,C_2,\hdots,C_n)$, the transmission line resistance matrix $R_t = \textsf{diag}(r_{t1}, r_{t2},\hdots,r_{tm})$ and transmission line inductance matrix $L_t =\textsf{diag}(l_{t1}, l_{t2},\hdots,l_{tm})$.

Hence, the state space dynamics of the entire system in equation (\ref{dq0_eqn}) is given as:
\begin{equation}
\begin{aligned}\label{ss_cont}
& \dot{\textbf{x}}(t)=\mathrm{A\textbf{x}(t)}+\mathrm{B\textbf{u}(t)}+\mathrm{E\textbf{i}_{L{dq}}(t)}, \\
& \mathrm{\textbf{y}(t)}=\mathrm{C\textbf{x}(t)},
\end{aligned}
\end{equation}
where $\textbf{x}(t) = [\mathbf{v}_d^\top, \mathbf{v}_q^\top, \mathbf{i}_d^\top, \mathbf{i}_q^\top,\mathbf{i}_{td}^\top,  \mathbf{i}_{tq}^\top]^\top  \in \mathbb{R}^{(4n+2m)}$ is the state vector, $\textbf{u}(t)=[\mathbf{u}_d^\top,  \mathbf{u}_q^\top]^\top \in \mathbb{R}^{2n}$ is the input vector and $\mathbf{y}(t)= [\mathbf{x}_d^\top,  \mathbf{x}_q^\top]^\top \in \mathbb{R}^{2n}$ is the output vector. $A \in \mathbb{R}^{(4n+2m)\times (4n+2m)}$ is the state matrix of the system, $B \in \mathbb{R}^{(4n+2m)\times (2n)}$, $C \in \mathbb{R}^{(2n)\times (4n)}$ is the measurement or output matrix and $E \in \mathbb{R}^{(4n+2m)\times (2n)}$ is the disturbance matrix resulting from the bounded load current demands at the AC buses. 

\begin{figure}[h]
    \centering \includegraphics[width=0.5\textwidth]{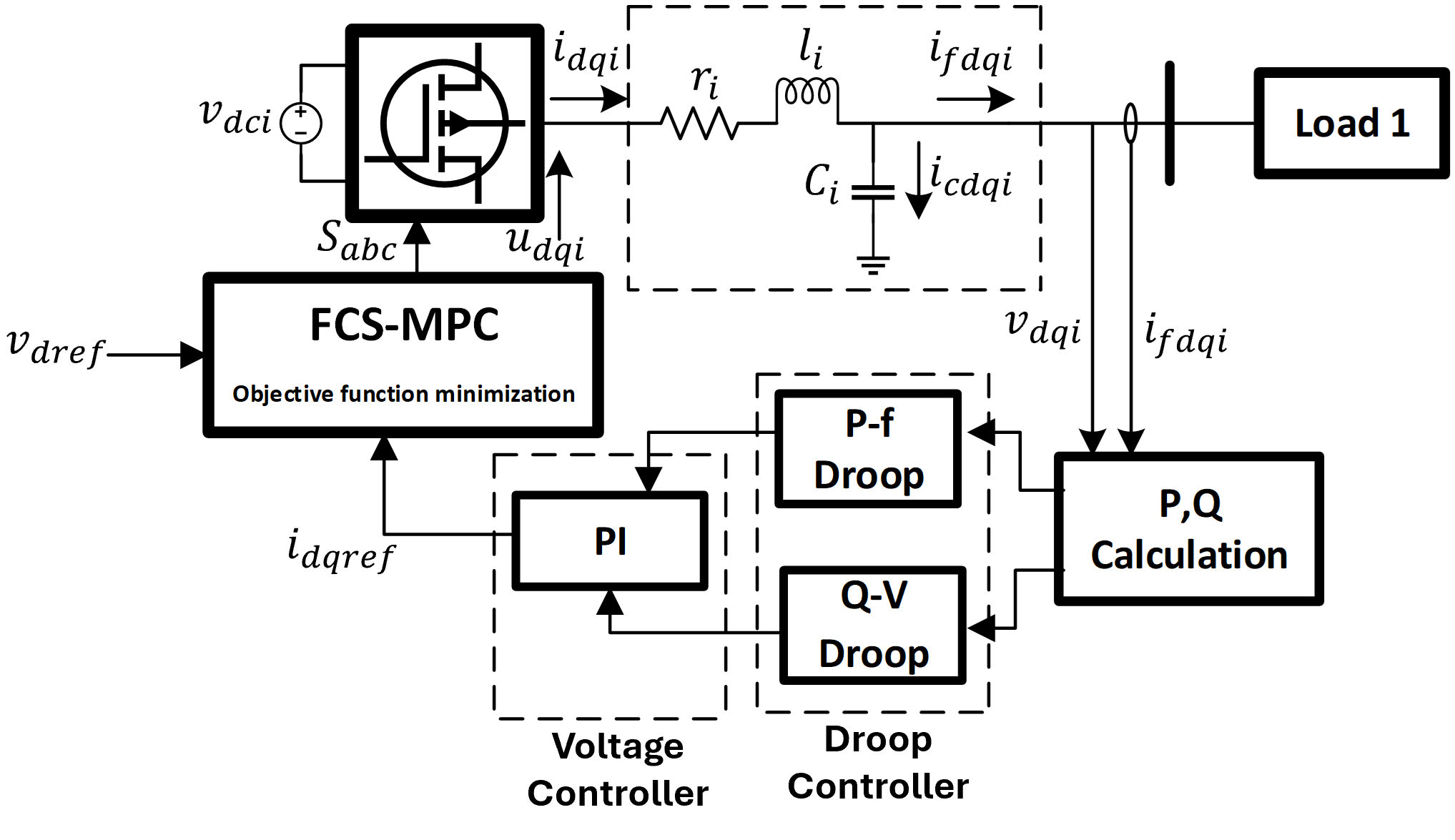}
    \caption{Droop-based FS-MPC Control Diagram}
    \label{fig:figb}
\end{figure}
\section{Control Development}
Figure~\ref{fig:figb} shows the control diagram of the i-th IBR supposedly connected to a j-th neighboring inverter through the transmission lines in the set $\mathcal{E} \in \mathcal{V} \times \mathcal{V} = \{1,....,m\}$. The filter output voltage, which is the same as the capacitor voltage $\mathbf{v}_{dq}$ and the output current of the filter $\mathbf{i}_{fdq}$ are used to compute the active and the reactive power references that are fed into the droop controller. The droop control controller provides  the frequency and the voltage required by the PI controller in the voltage control loop to compute the current reference $\mathbf{i}_{ref}$ for the FCS-MPC, which in turn provides the optimum PWM signals $S_{abc}$ to switch the IBRs.

The control objectives in this paper are to ensure perfect tracking of the IBR filter output voltage, AC grid frequency regulation, active and reactive power sharing and the stability analysis for the FCS-MPC. 
The next subsections detail the droop control development for power sharing and the Finite-set model predictive control formulation for the filter output voltage and current tracking.

\subsection{Droop Control development}
According to \cite{laaksonen2005voltage}, if $\delta_{ij}$ is very small for a highly inductive transmission line. That is, $\sin \delta_{ij}\approx\delta_{ij}$ and $\cos \delta_{ij}\approx 1$, 
where $\delta_{ij}$ is the phase shift angle, $v_i$ and $v_j$ are the output voltages of each inverter and $X_j$ is the inductive reactance of the transmission line. Then, equation for the output active power (P) and reactive power (Q) developed between an IBRi connected to a neighboring IBRj can be rewritten as
\begin{equation}
\begin{aligned}\label{PQ-simplified}
P &=\frac{v_i v_j}{X_j} \delta_{ij},\\
Q &=v_j\bigg(\frac{v_i-v_j}{X_j}\bigg).
\end{aligned}
\end{equation}

From (\ref{PQ-simplified}), the control of the voltage difference $\mathbf{v}_i-\mathbf{v}_j$ depends on the reactive power Q and that of the phase angle which is dynamically equivalent to the grid frequency, depends on the active power P. Hence, droop control for several IBRs is given by
\begin{equation}
\begin{aligned}\label{v-w_droop}
{v}_i &= {v}_{nom,i}-n_{qi} Q_i, \\
\omega_i  &=\omega_{g}-m_{pi} P_i,
\end{aligned}
\end{equation}
where $v_{nom,i}$ is the nominal voltage rating of each IBR that must be met for perfect voltage regulation, $\omega_g$ is a set of constant dq rotating reference frame frequency, $n_{qi}$ is the reactive power-voltage $(Q-V)$ droop coefficient of each IBR, $m_{pi}$ is the active power-frequency $(P-\omega)$ droop coefficient of each IBR and
\begin{equation}
\begin{aligned}\label{PQ_powers}
P_i &= \frac{3}{2}\mathbf{v}_{dqi}^T\mathbf{i}_{dqi},\\
Q_i &= \frac{3}{2}\mathbf{v}_{dqi}^TJ\mathbf{i}_{dqi}.
\end{aligned}
\end{equation}

The droop coefficients must be carefully selected or computed because they affect stability and the way loads are being shared between a set of inverter-based resources (IBRs)\cite{bayhan2016model}. Thus, for the total power P to be shared among n inverters, the load distribution between them must satisfy the following:
\begin{equation}
\begin{aligned}
m_{p1} P_1=m_{p2} P_2=\hdots&=m_{pn} P_n,\\
P_1+P_2+P_3+\hdots+P_n &= P.
\end{aligned}
\end{equation}

When the grid synchronization is perfectly matched at a steady state, active power is efficiently shared between IBRi and IBRj as long as the condition in (11) is satisfied.
\begin{equation}
\frac{\textbf{i}_{di}}{\textbf{i}_{dj}}=\frac{m_{pj}}{m_{pi}} \quad \forall i, j \in \mathcal{V} .
\end{equation}
That is using $P_i$ in equations (\ref{v-w_droop}), (\ref{PQ_powers}), and assuming perfect synchronization at steady state, $\omega_i  =\omega_g\mathbf{1}$. Thus, 
\begin{equation}\label{P_sync}
\frac{P_{i}}{P_{j}}=\frac{\frac{3}{2}\mathbf{v}_{dqi}^T\mathbf{i}_{dqi}}{\frac{3}{2}\mathbf{v}_{dqj}^T\mathbf{i}_{dqj}}=\frac{\mathbf{v}_{di} \mathbf{i}_{di}+\mathbf{v}_{qi} \mathbf{i}_{qi}}{\mathbf{v}_{dj} \mathbf{i}_{dj}+\mathbf{v}_{qj} \mathbf{i}_{qj}}=\frac{m_{pj}}{m_{pi}} \quad \forall i, j \in \mathcal{V} .
\end{equation}
Since voltage does not vary much in a voltage-regulated power grid and the power factor is almost unity such that $\textbf{v}_d$ component of the voltage is much higher than the $\textbf{v}_q$ component. Hence,
\begin{equation}
\frac{\mathbf{i}_{di}}{\mathbf{i}_{dj}}=\frac{m_{pj}}{m_{pi}} \quad \forall i, j \in \mathcal{V} .
\end{equation}
\begin{remark}\label{remark-2}
\textit{We assume that the AC microgrid is well-regulated, perfectly synchronized, balanced, and has a power factor approximately equal to 1}. 
\end{remark}
\subsection{Finite Control Set MPC Formulation}
To develop the FCS-MPC algorithm for the state space dynamics of the decentralized system shown in (3). In discrete state space form, the dynamics in (\ref{ss_cont}) then becomes
\begin{equation}
\begin{aligned}\label{ss_disc}
 \textbf{x}(k+1)&=\mathrm{A_d}\textbf{x}(k)+\mathrm{B_d}\textbf{u}(k) +\mathrm{E_d\textbf{i}_{L{dq}}(k)}, \\
\mathrm{\textbf{y}(k)}&=\mathrm{C_d}\textbf{x}(k),
\end{aligned}
\end{equation}
where $\mathrm{A_d}=I+\mathrm{A}\tau_s$, $\mathrm{B_d}=\mathrm{B}\tau_s$, $C_d=C$ and $\mathrm{E_d}=\mathrm{E}\tau_s$. The dimensions of the state matrices and the state variables remain the same except for the discrete input matrix $B_d$ and the control inputs $\textbf{u}(k)$ in (\ref{ss_disc}). The control input now belongs to the finite set of inputs $\mathcal{U}$. That is:
\begin{equation}\label{input_set}
 \textbf{u}(k) \in \mathcal{U}=\left\{u_1, u_2, \ldots, u_M\right\} \subset \mathbb{R}^{2n}
\end{equation}
Since a modulator is not required for PWM generation, the decision variables for FCS-MPC are the switching states ($S_a, S_b, S_c$) in each phase of the 3-phase, 2-level inverters used in this system and the relationship between the switching state vector $S_{abc}$ and the finite number of control variables $\textbf{u}(k)$ is given as
\begin{equation}
\begin{aligned}\label{switching_S}
S_{abc} &= \frac{2}{3}(S_a+S_be^{j\frac{2\pi}{3}}+S_ce^{j\frac{4\pi}{3}}) \\
\mathbf{u}_i(k)&=V_{dc}S_{abc}   \quad \forall i \in \{1,2, \hdots, M\} .
\end{aligned}
\end{equation}
The control objective here is to regulate the voltage $\textbf{v}_d$ across the filter capacitor and track the injected current provided by the voltage loop using the following optimization problem:
\begin{equation}
\begin{aligned}
\Minimize_{\mathbf{v},\mathbf{i}} \quad & \left\|\mathbf{y}_\text{ref}-C\mathbf{x}(k+1)\right\|^2\\
\textrm{s.t.} \quad & \textbf{x}(k+1)=\mathrm{A_d}\textbf{x}(k)+\mathrm{B_d}\textbf{u}(k) +\mathrm{E_d\textbf{i}_{L{dq}}(k)}, \\
\quad & \mathbf{u}(k) \in \mathcal{U},
\end{aligned}
\end{equation}
where $\textbf{y}(k) = [\mathbf{v}_d^\top,  \mathbf{v}_q^\top,\mathbf{i}_d^\top,  \mathbf{i}_q^\top]^\top$ and $C \in \mathbb{R}^{{2n} \times {4n}}$.
The reference currents $\textbf{i}_{d q r e f}(k)=[\textbf{i}_{dref}(k)^\top,  \textbf{i}_{qref}(k)^\top]^\top$ for the FCS-MPC controller is generated from the filter dynamic by using PI controller in the voltage control loop as shown in (\ref{idqref_input}) \cite{8557936}.
\begin{equation}
\begin{split}\label{idqref_input}
     \textbf{i}_{d q r e f}=\textbf{i}_{d q}-J \omega_g C_i \textbf{v}_{d q}+K_p\left(\textbf{v}_{d q r e f}-\textbf{v}_{d q}\right)+\\
     K_i \int_{0}^{\tau}\left(\textbf{v}_{d q r e f}-\textbf{v}_{d q}\right) d \tau
\end{split}  
\end{equation}

Reference voltages $\textbf{v}_{d q r e f}(k)=[\textbf{v}_{dref}(k)^\top,  \textbf{v}_{qref}(k)^\top]^\top$ is the grid rated voltage of $220\mathrm{~V}$ with $\textbf{v}_{qref}=0$, as mentioned in remark \ref{remark-2}. And PI controller gains $K_p$ and $K_i$ are carefully selected. Then, since AC grid reference voltage doesn't change much at every sampling time and the inverter switching frequency is much higher than the AC grid frequency, we can approximate both the voltage and the current references as $\textbf{v}_{dref}(k+1)\approx \textbf{v}_{dref}$ and $\textbf{i}_{dref}(k+1)\approx \textbf{i}_{dref}$.

\section{Numerical Simulation Results}
 We implement the proposed Droop-based FCS-MPC control algorithm on three IBRs ($n =3$) with the current flow through two transmission lines ($m = 2$) connecting them as shown in Figure \ref{fig:fig3}, where the incidence matrix $\mathcal{B}_{i k} \in 
 \mathbb{R}^{3\times2}$ is given as:
 
\begin{equation*}
\mathcal{B}_{i k}=\left[\begin{array}{cc}
1 & 0 \\
-1 & 1 \\
0 & -1
\end{array}\right]
\end{equation*}
\begin{figure}[h]
    \centering \includegraphics[width=0.32\textwidth]{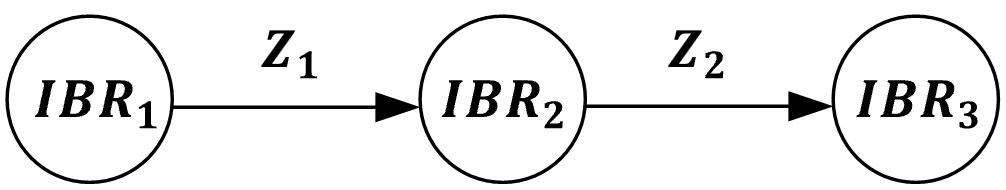}
    \caption{Current flow between connected IBRs}
    \label{fig:fig3}
\end{figure}

In our numerical simulation, we assume a homogeneous configuration of IBRs, and the parameters are given in Table 1. The transmission line impedances are $Z_{1}$ connecting IBR1 to IBR2 and  $Z_{2}$ connecting IBR2 to IBR3 as shown in Figure \ref{fig:fig3}. $Z_{1}$ and $Z_{2}$ have resistances $r_{t1} =r_{t2} =0.2 \Omega$ and inductances $l_{t1}=l_{t2} = 0.3\mathrm{~mH}$. Constant impedance loads are available at buses 1, 2 and 3 as $S_{L1}= S_{L2}=1+j0.2 \mathrm{~kVA}$, $S_{L3}= 3+j0.2 \mathrm{~kVA}$. Additional load of $S_{L4}= 2+j0.2 \mathrm{~kVA}$ is suddenly added to bus 3 at $t = 1s$ and later disconnected at $t = 3s$.

\begin{table}[h!]
\caption{parameters for the IBRs}
\centering
\begin{tabular}{ll}
\hline Parameters & Value \\
\hline DC Source Voltage,$V_{dc}$ & $600\mathrm{~V}$ \\
Rated Voltage (rms) & $220 \mathrm{~V}$ \\
Grid frequency, $\omega_g/(2\pi)$ & $50 \mathrm{~Hz}$ \\
Switching frequency, $f_s$ & $10 \mathrm{~kHz}$ \\
Filter inductor, $l_i$ & $3.5\mathrm{~mH}$ \\
Filter capacitor, $C_i$ & $50\mathrm{~\mu F}$ \\
Filter resistor, $r_i$ & $0.2\mathrm{~\Omega}$ \\
P-f droop coefficients, $m_{p1}=m_{p2}=m_{p3}$ & $6e^{-5}\mathrm{~rad/s/W}$\\
Q-V droop coefficients, $n_{q1}=n_{q2}=n_{q3}$ & $4e^{-3}\mathrm{~V/Var}$ \\
\hline
\end{tabular}
 \end{table}

\begin{figure}[h!]
    \centering \includegraphics[width=0.7\textwidth]{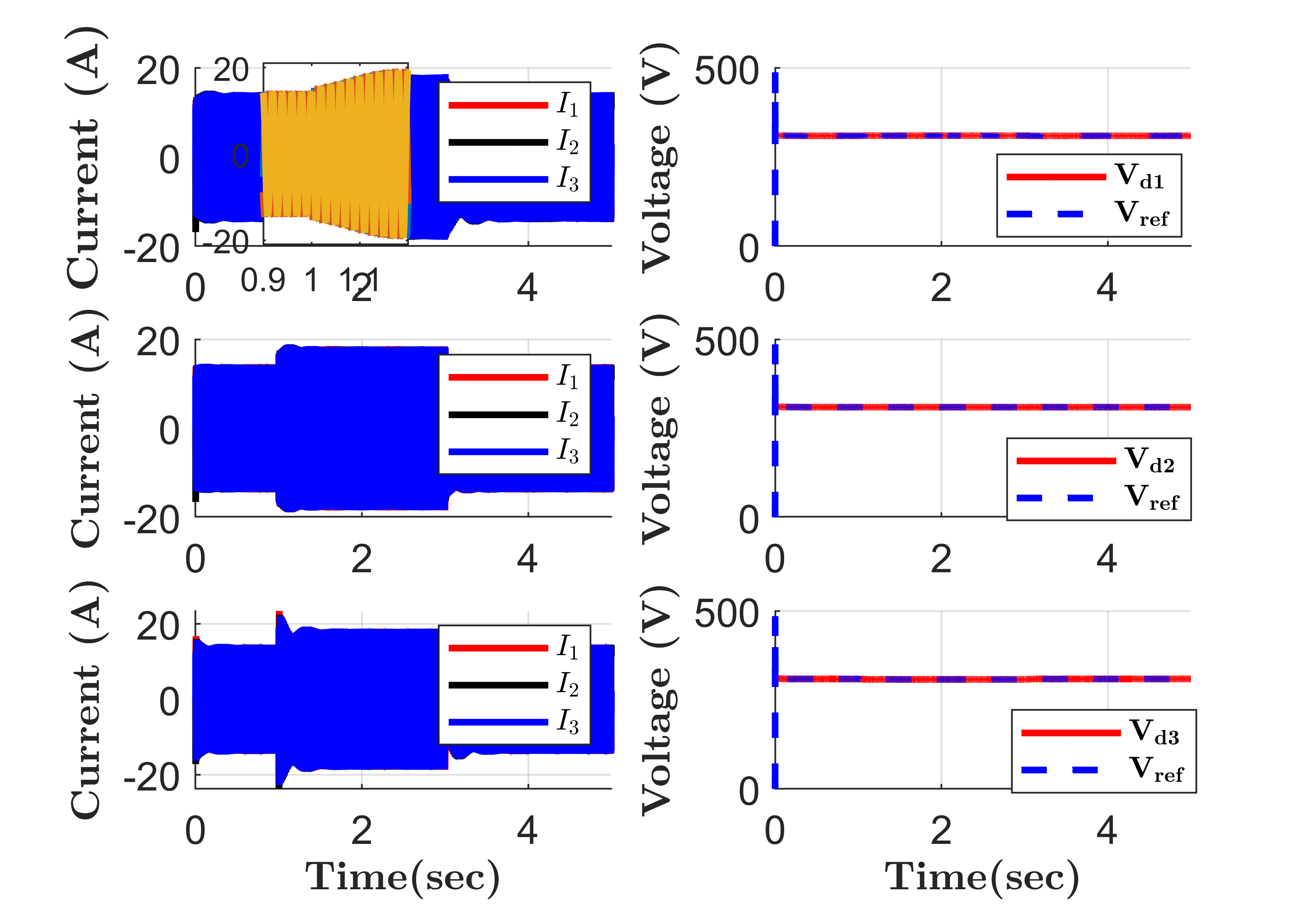}
    \caption{(a) Response of 3-phase currents to load changes and (b) Filter capacitor voltage tracking for IBR1, IBR2 and IBR3.}
    \label{3ph_currents}
\end{figure}

\begin{figure}[t!]
    \centering \includegraphics[width=0.7\textwidth]{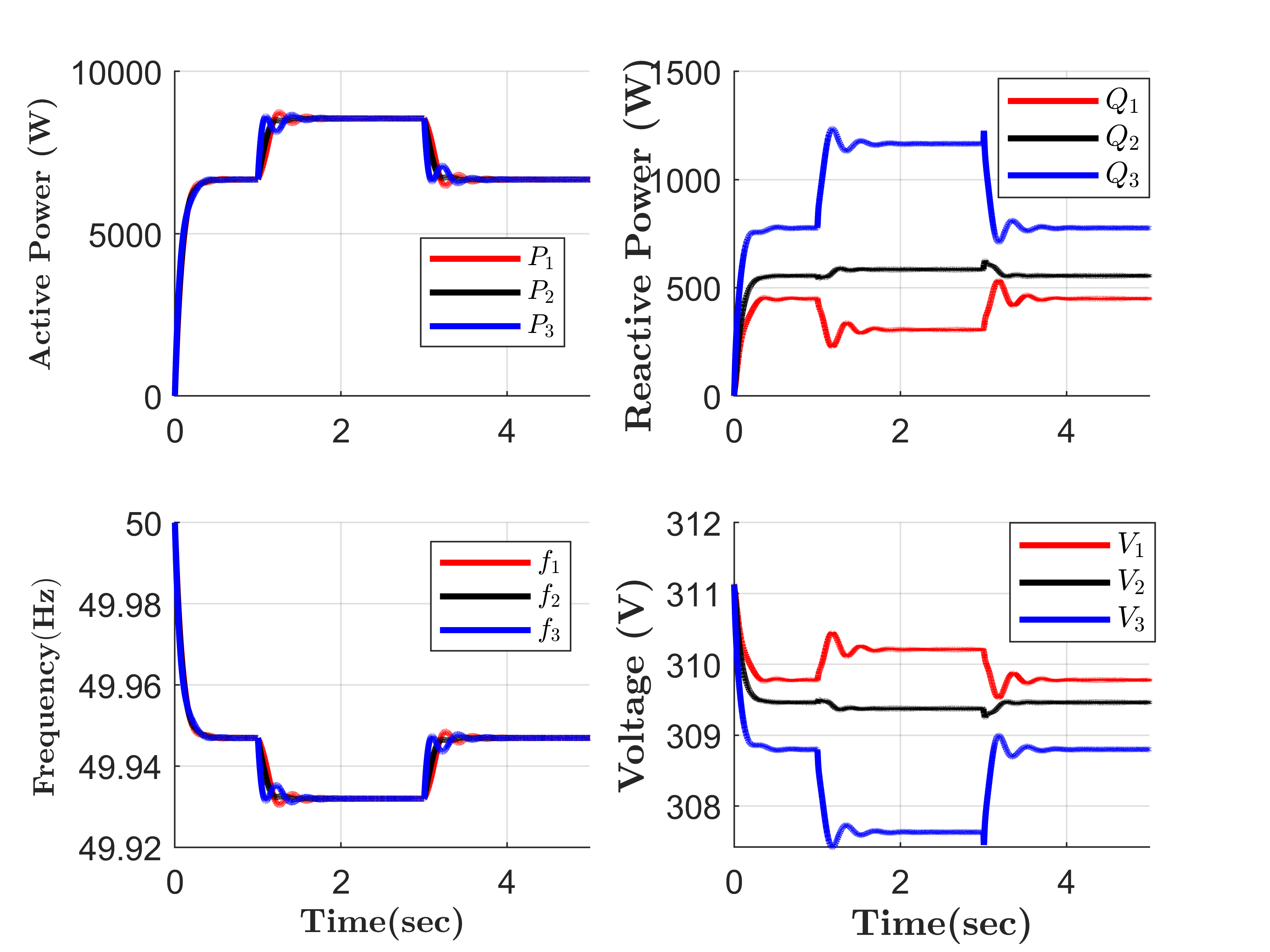}
    \caption{FCS-MPC-based Droop for IBR1, IBR2 and IBR3: 
    (a)Active-Power/Frequency (b) Reactive-power/Voltage}
    \label{active_P_w}
\end{figure}

Figure \ref{3ph_currents}(a) shows that 3-phase currents in each IBR change when the load changes. That is, the currents increase when load is added at bus 3, and return to its initial value when the load is disconnected. Figure \ref{3ph_currents}(b) shows that the output voltage across the capacitor does not change when the load is disconnected and reconnected.  The voltage remains at a peak value of $311 \mathrm{V}(220 \mathrm{V_{rms}})$ regardless of the change in load. This confirms that the FCS-MPC provides effective voltage regulation across the filter capacitor at the AC buses. 

Figure \ref{active_P_w}(a) shows that efficient frequency synchronization and active power sharing with fewer oscillations are achieved at a steady state. Frequency deviation is minimal and reaches a steady state quickly during connection and disconnection of load at bus 3. Active power is equally shared proportional to the stiff frequency maintained in the AC grid. This shows the effectiveness of the decentralized droop-based FCS-MPC control for active power sharing. Slight frequency deviation is normal based on (\ref{v-w_droop}) but this can be further improved by a higher-level controller which is not considered here. Figure \ref{active_P_w}(b) shows the reactive power sharing for the output voltages of IBR1, IBR2, and IBR3. Reactive power changes with respect to changes in load and this deviation is more pronounced at bus 3 where additional load is added. Although the inverter voltages do not deviate significantly from each but the reactive is more pronounced at higher load changes. Droop-based FCS-MPC reduces chattering and oscillations in reactive power sharing and the slight deviation is due to the grid impedance mismatch \cite{zhong2011robust}. And has been addressed in literature by using virtual impedance droop control \cite{Anubi_2022, he2011analysis}. Unlike the P-f droop, the Q-V droop is local to each inverter depending on how long they are connected to each other with respect to line impedance variations. However, the reactive powers and the voltages of the IBRs reach a steady state quickly as the load connects and disconnects at bus 3 efficiently. This shows the effectiveness of FCS-MPC in the inner control loop. 

\section{Conclusion}

A Droop-based finite set model predictive control technique is proposed to regulate the voltage across the filter capacitor at the load bus of an islanded AC microgrid. This method provides a PWM signal to control the inverters directly without the need for a modulator. Accurate frequency synchronization and effective active power sharing are achieved at a steady state with fewer oscillations. Filter capacitor voltage at the AC bus remains unchanged despite the load changes as depicted in the numerical simulations thereby confirming the affirmation in Figure 4(b). Reactive power sharing is also achieved with fewer oscillations and low inverter voltage variations. This method eliminates the need for tuning parameter gains in the inner current loop controller, making it faster and the slight frequency deviation at steady state can be further improved by introducing a secondary or supervisory controller in future work. Also, we developed the condition for bounded stability for FCS-MPC.

\bibliographystyle{IEEEtran}
\bibliography{myreferences}

\end{document}